\documentclass[preprint]{elsarticle}
\usepackage{hyperref}
\usepackage{xcolor}
\usepackage[utf8]{inputenc}
\usepackage[T1]{fontenc}
\usepackage{amsmath,amsfonts,amsthm,amssymb,dsfont,bbm}
\theoremstyle{plain}
\newtheorem{theorem}{Theorem}
\newtheorem{lemma}{Lemma}
\newtheorem{proposition}{Proposition} 
\newtheorem{definition}{Definition}
\newtheorem{corollary}{Corollary}

\theoremstyle{definition}

\newtheorem{remark}{Remark}

\newcommand{\rmd}{\mathrm{d}}
\newcommand{\bbE}{\mathbb{E}}
\newcommand{\bbR}{\mathbb{R}}

\journal{International Journal of Forecasting}

\biboptions{authoryear}
\newcommand*{\citetnp}[1]{\citeauthor{#1}, \citeyear{#1}}

\begin{document}
\newcommand{\tnotemarksymbol}{\dag}
\begin{frontmatter}
	
	\title{Distributional regression and its evaluation with the CRPS : bounds and convergence of the minimax risk\\
    \small \textcolor{blue}{\textit{The following article is a preprint version of the article. You can find the latest version of the article here :} \url{https://doi.org/10.1016/j.ijforecast.2022.11.001}.}}
	
	\author[lmb]{Romain Pic\corref{correspondingauthor}}
	\cortext[correspondingauthor]{Corresponding author}
	\ead{romain.pic@univ-fcomte.fr}
	\author[lmb]{Clément Dombry}
	\author[lsce]{Philippe Naveau}
	\author[cnrm,meteofr]{Maxime Taillardat}
	
	\date{01/12/2022}

	\address[lmb]{Laboratoire de Mathématiques de Besançon, CNRS UMR 6623, Univ. Bourgogne Franche-Comté, Besançon, France}
	\address[lsce]{Laboratoire des Sciences du Climat et de l'Environnement, UMR 8212, CEA-CNRS-UVSQ, IPSL \& U Paris-Saclay, Gif-sur-Yvette, France}
	\address[cnrm]{CNRM, Université de Toulouse, Météo-France, CNRS, Toulouse, France}
	\address[meteofr]{Météo-France, Toulouse, France}

	\begin{abstract}
		The theoretical advances on the properties of scoring rules over the past decades have broadened the use of scoring rules in probabilistic forecasting. In meteorological forecasting, statistical postprocessing techniques are essential to improve the forecasts made by deterministic physical models. Numerous state-of-the-art statistical postprocessing techniques are based on distributional regression evaluated with the Continuous Ranked Probability Score (CRPS). However, theoretical properties of such evaluation with the CRPS have solely considered the unconditional framework (i.e. without covariates) and infinite sample sizes. We extend these results and study the rate of convergence in terms of CRPS of distributional regression methods. We find the optimal minimax rate of convergence for a given class of distributions and show that the $k$-nearest neighbor method and the kernel method reach this optimal minimax rate.
	\end{abstract}
	

\end{frontmatter}

\section{Introduction}

In meteorology, ensemble forecasts are based on a given number of deterministic models whose parameters vary slightly in order to consider observation errors and incomplete physical representation of the atmosphere. This leads to an ensemble of different forecasts that overall also assess the uncertainty of the forecast. Ensemble forecasts suffer from bias and underdispersion \citep{Hamill_1997,Baran_2018} and need to be statistically postprocessed in order to be improved. Different postprocessing methods have been proposed, such as Ensemble Model Output Statistics \citep{Gneiting_2005},  Quantile Regression Forests \citep{Taillardat_2016} or Neural Networks \citep{Schulz_2022}. These references, among other, also discuss the stakes of weather forecast statistical postprocessing.\\

Postprocessing methods rely on distributional regression \citep{Gneiting_2014} where the aim is to predict the conditional distribution of the  quantity of interest (e.g. temperatures, wind-speed, or precipitation) given  a set of covariates (e.g. raw outputs of a physical ensemble model). Algorithms are often based on the minimization of a proper scoring rule that compares actual observations with the predictive distribution. Scoring rules can be seen as an equivalent of loss functions in classical regression. A detailed review of scoring rules is given by \citet{Gneiting_2007}. The Continuous Ranked Probability Score (CRPS; \citetnp{Matheson_1976}), defined in Equation \eqref{eq:CRPS_def}, is one of the most popular scores in meteorological forecasting. The CRPS is also minimized to infer parameters of statistical models used in postprocessing (e.g. \citetnp{Gneiting_2005}; \citetnp{Naveau_2016}; \citetnp{Rasp_2018}; \citetnp{Taillardat_2019}). Recently, under monotonicity assumptions, the isotonic distributional regression \citep {Henzi_2021} was shown to minimize the in-sample CRPS and to satisfy  consistency in the sense of Kolmogorov distance.\\

To the best of our knowledge, most convergence statements in distributional regression (e.g. \citetnp{Thorey_2017} and \citetnp{Mosching_2020}) are not only derived within an unconditional framework, i.e. without taking into account the covariates, but also these limiting results assume arbitrarily large sample sizes. In this work, our goal is to bypass these two limitations.\\

This paper is organized as follows. Section \ref{section:prelim} introduces preliminary notions that are needed to state our main results  in Section~\ref{section:main_res}. Section~\ref{subsection:distrib_regression} introduces our framework and notation for distributional regression. 
Section~\ref{subsection:CRPS_distrib} provides the theoretical background on distributional regression and its evaluation using the CRPS and Section~\ref{subsection:prelim-minimax} provides some elements on minimax risk theory. Section \ref{subsection:knn_kernel_distrib} briefly introduces the two models that are studied in this article: the $k$-nearest neighbor and kernel estimators.  The main result on minimax rate of convergence for distributional regression is stated in Section~\ref{subsection:optimal-minimax} where suitable classes of distributions $\mathcal{D}^{(h,C,M)}$ are defined. In Section \ref{subsection:knn}, we study the $k$-NN estimators and derive a non-asymptotic upper bound for the excess risk of the CRPS uniformly on the class $\mathcal{D}^{(h,C,M)}$. Section \ref{subsection:kernel} provides similar results for the kernel method. In Section \ref{subsection:lower-minimax}, we find a lower minimax rate of convergence  by reducing the problem to standard point regression solved by \citet{Gyorfi_2002}. We can deduce that the $k$-NN  method for the distributional regression reaches the optimal rate of convergence in dimension $d\geq2$, while the kernel method reaches the optimal rate of convergence in any dimension. All the proofs are postponed to and detailed in the Appendix. A short conclusion and discussion is provided in Section~\ref{section:conclusion}.

\section{Preliminaries}\label{section:prelim}

\subsection{Distributional regression framework}\label{subsection:distrib_regression}
In this article, we consider the regression framework $(X,Y)\in\bbR^d\times\bbR$ with distribution $P$. The goal of distributional regression is to estimate the conditional distribution of $Y$ given $X=x$, noted  
\[
F^\ast_x(y)\colon=P(Y\leq y|X=x),\quad x\in\bbR^d.
\]

In forecast assessment, we make the distinction between the construction of the estimator relying on the training sample $D_n=\{(X_i,Y_i),1\leq i\leq n\}$ and its evaluation with respect to new data $(X,Y)$.  Given  the training sample $D_n$, the forecaster constructs a predictor $\hat F_n:x\mapsto \hat F_{n,x}$ that estimates the conditional distribution $F^\ast_x$. In this context, it is crucial to assess if $\hat F_{n,x}$ is close to $F^\ast_x$ over the entire range of possible values of $X=x$. To this aim, we consider
\begin{equation}\label{eq:l2-distance}
	\bbE_{X\sim P_X,D_n\sim  P^n }\left[\int_\bbR | \hat F_{n,X}(z)-F^\ast_X(z)|^2 \rmd z\right]
\end{equation}
where  $P_X$ denotes the marginal distribution of $X$, $\bbE_{X\sim P_X,D_n\sim  P^n }$ denotes the expectation with respect to $X$ and $D_n$ following $P_X$ and $P^n$ respectively. The squared $L^2$-norm within the expectation is usually referred to as the squared second-order \textit{Cram\'er's distance}. We focus on this specific distance because it corresponds to the excess risk associated with the CRPS, also called \textit{divergence} of the CRPS, as explained in the next section.

\subsection{CRPS and evaluation of distributional regression}\label{subsection:CRPS_distrib}

    The Continuous Ranked Probability Score (CRPS; \citetnp{Matheson_1976}) compares a predictive distribution $F$ and a real-valued observation $y$ by computing the following integral
    \begin{equation}\label{eq:CRPS_def}
    	\mathrm{CRPS}(F,y)=\int_{\mathbb{R}}(F(z)-\mathds{1}_{y\leq z})^2\rmd z.
    \end{equation}
    The expected $\mathrm{CRPS}$ of a predictive distribution $F$ when the observations $Y$ are distributed according to $G$ is defined as
    \begin{equation}\label{eq:expected_CRPS}
    	\overline{\mathrm{CRPS}}(F,G)=\int_\bbR \mathrm{CRPS}(F,y) G(\rmd y),\quad F,G\in\mathcal{M}(\mathbb{R}),
    \end{equation}
    where $\mathcal{M}(\mathbb{R})$ denotes the set of all distribution functions on $\bbR$. This quantity is finite when both $F$ and $G$ have a finite first moment. Then, the difference between the expected CRPS of the forecast $F$ and the expected $\mathrm{CRPS}$  of the ideal forecast $G$ can be written as
    \begin{equation}\label{eq:div_CRPS}
    	\overline{\mathrm{CRPS}}(F,G)-\overline{\mathrm{CRPS}} (G,G) = \int_\bbR |F(z)-G(z)|^2 \rmd z \geq 0.
    \end{equation}
    This implies that the only optimal prediction, in the sense that it  minimizes the expected CRPS, is the true distribution $G$. A score with this property is said to be \textit{strictly proper}. This property is essential for distributional regression as it justifies the minimization of the expected score in order to construct or evaluate a prediction.\\
  
    In distributional regression, the quality of a predictor $\hat F:x\mapsto \hat F_x$ is assessed by its risk 
    \begin{align*}
    	R_P(\hat F) &= \bbE_{(X,Y)\sim P}\left[\mathrm{CRPS}(\hat F_X,Y)\right]\\
    	&= \bbE_{X\sim P_X} \left[\overline{\mathrm{CRPS}}(\hat F_X,F_X^\ast) \right].
    \end{align*}
    This quantity is important as many distributional regression methods try to minimize it in order to improve predictions. When $Y$ is integrable, Equation~\eqref{eq:div_CRPS} implies
    \begin{align}
    	R_P(\hat F) - R_P(F^\ast) &= \bbE_{(X,Y)\sim P}\left[\mathrm{CRPS}(\hat F_X,Y)-\mathrm{CRPS}( F^\ast_X,Y)\right]\nonumber\\
    	&= \bbE_{X\sim P_X}\left[\int_\bbR \left|\hat F_X(z)-F^\ast_X(z)\right|^2 \rmd z \right]\geq 0. \label{eq:bayes}
    \end{align}
    We recall that the Bayes risk is the minimal theoretical risk over all possible predictors and that a Bayes predictor is a predictor achieving the Bayes risk. Thus, Equation~\eqref{eq:bayes} implies that $R_P(F^\ast)$ is the Bayes risk and that $F^\ast$ is a Bayes predictor if and only if $\hat F_x=F^\ast_x$ $P_X$-a.e. An introduction to the notions of theoretical risk, Bayes risk and excess risk can be found in Section~2.4 of \citet{Hastie_2009}.
    
    Finally, we consider the case of a predictor $\hat F_n$ built on a training sample $D_n=\{(X_i,Y_i),1\leq i\leq n\}$, as presented in Section~\ref{subsection:distrib_regression}, to estimate the conditional distribution of $Y$ given $X$. Then, $(X,Y)$ denotes a new independent observation used to evaluate the performances of $\hat F_n$. The predictor has the expected $\mathrm{CRPS}$ 
    \begin{equation*}
    	\bbE_{D_n\sim P^n}[R_P(\hat F_n)]= \bbE_{D_n\sim P^n,(X,Y)\sim P}[\mathrm{CRPS}(\hat F_{n,X}, Y)],
    \end{equation*}
    with expectation taken both with respect to the training sample $D_n$ and test observation $(X,Y)$.
    Once again, when $Y$ is integrable,  the theoretical risk has a unique minimum given by $R_P(F^\ast)$. The \textit{excess risk} becomes
    \begin{align}
    	&\quad \bbE_{D_n\sim P^n}\left[R_P(\hat F_n)\right]-R_P(F^\ast) \nonumber \\
    	&= \bbE_{D_n\sim P^n,X\sim P_X}\left[\int_\bbR \left|\hat F_{n,X}(z)-F_X^\ast(z)\right|^2\rmd z\right]\geq 0.\label{eq:excess_risk}
    \end{align}
    This justifies the choice of the squared Cram\'er's distance in Equation~\eqref{eq:l2-distance}.
    
     For large sample sizes, one expects that the predictor correctly estimates the conditional distribution and that the excess risk \eqref{eq:excess_risk} tends to zero.  A genuine question is to investigate the rate of convergence of the excess risk to zero as the sample size $n\to\infty$. The risk depends on the distribution of observations and we want the model to perform well on large classes of distributions. Hence, we consider the standard minimax approach, as described in the next section.
     
    \subsection{Optimal minimax rates of convergence}\label{subsection:prelim-minimax}
    In order to study the rate of convergence,  as $n\to\infty$, of the excess risk \eqref{eq:excess_risk} to zero, we introduce the notion of \textit{optimal minimax rate of convergence}. The minimax risk corresponds to the best achievable risk in the worst-case scenario (whence the name minimax). More precisely,  given a class of distributions $\mathcal{D}$,  the optimal minimax rate of convergence quantifies the minimal error that an estimator $\hat F_n$ can achieve uniformly on a given class of distributions $\mathcal{D}$,  when the size of the training set $D_n$ gets large.
    
    \citet{Stone_1982} provided minimax rates of convergence within a point regression framework and the minimax theory for nonparametric regression is well-developed, see e.g. \cite{Gyorfi_2002} or \cite{Tsybakov_2009}. To the extent of our knowledge, this paper states the first results for distributional regression.

    The formal definition of minimax rate of convergence for distributional regression is as follows.
    \begin{definition}
    	A sequence of positive numbers $(a_n)$ is called an optimal minimax rate of convergence on the class $\mathcal{D}$ if
    	\begin{equation}\label{eq:lowerbound}
    		\liminf_{n\to\infty} \inf_{\hat F_n} \sup_{P\in \mathcal{D}} \cfrac{\bbE_{D_n\sim P^n}[R_P(\hat F_n)]-R_P(F^\ast)}{a_n}>0
    	\end{equation}
    	and
    	\begin{equation}\label{eq:upperbound}
    		\limsup_{n\to\infty} \inf_{\hat F_n} \sup_{P\in \mathcal{D}} \cfrac{\bbE_{D_n\sim P^n}[R_P(\hat F_n)]-R_P(F^\ast)}{a_n}<\infty,
    	\end{equation}
    	where the infimum is taken over all distributional regression models $\hat F_n$ trained on $D_n$. If the sequence $(a_n)$ satisfies only the lower bound~\eqref{eq:lowerbound}, it is called a lower minimax rate of convergence.
    \end{definition}
    
    \subsection{$k$-NN and kernel predictors in distributional regression}\label{subsection:knn_kernel_distrib}

    Many predictors $\hat F_{n}$ can be studied and possibly achieve the optimal minimax rate of convergence. In this paper, we focus on two simple cases: $k$-nearest neighbor and kernel estimators.
    
    The $k$-nearest neighbor ($k$-NN) method is well-known in the classical framework of regression and classification (see, e.g. \citetnp{Biau_2015}). In distributional regression, the $k$-NN method can be suitably adapted to estimate the conditional distribution $F^\ast_x$ and the estimator is written as
    \begin{equation}\label{eq:knn-model}
    	\hat F_{n,x}(z)= \frac{1}{k_n}\sum_{i=1}^{k_n} \mathds{1}_{Y_{i:n}(x)\leq z},
    \end{equation}
    where $1\leq k_n\leq n$ and $Y_{i:n}(x)$ denotes the observation at the $i$-th nearest neighbor of $x$. As usual, possible ties are broken at random to define nearest neighbors. Note that, in weather forecast statistical postprocessing, the $k$-NN method corresponds to a type of analog ensemble method (see \citetnp{Delle_Monache_2013}).
    
    The kernel estimate in distributional regression (see, e.g. Chapter 5 of \citetnp{Gyorfi_2002}) can be expressed as
    \begin{equation}\label{eq:kernel-gen}
    	\hat F_{n,x}(z)= \cfrac{\sum_{i=1}^{n} K(\frac{x-X_i}{h_n})\mathds{1}_{Y_i\leq z}}{\sum_{i=1}^{n}K(\frac{x-X_i}{h_n})},
    \end{equation}
    where the function  $K:\bbR^d \to[0, \infty)$ is a density function, called kernel, and $h_n>0$ is the so-called bandwidth, that depends on the sample size $n$. If the denominator in \eqref{eq:kernel-gen} vanishes, we use the convention  $\hat F_{n,x}(z)=\frac{1}{n}\sum_{i=1}^n \mathds{1}_{Y_i\leq z}$.
    
    Minimax rates of convergence of the $k$-NN and kernel models in point regression are well-studied and it is known that, for  suitable choices of number of neighbors $k_n$ and bandwidth $h_n$ respectively, the methods are minimax rate optimal on classes of distributions with Lipschitz or more generally Hölder continuous regression functions (see e.g. Theorem 14.5 in \citetnp{Biau_2015} and Theorem 5.2 in \citetnp{Gyorfi_2002}). For suitable classes of distributions defined hereafter, we are able to extend these results to distributional regression. Moreover, we obtain non-asymptotic bounds for the minimax rate of convergence for both the $k$-NN and kernel models (see Sections \ref{subsection:knn} and \ref{subsection:kernel}).

   \section{Main results}\label{section:main_res}
   \subsection{Optimal minimax rate of convergence}\label{subsection:optimal-minimax}
 We consider the following classes of distributions. 
    \begin{definition}\label{def:classD}
    	For $h\in (0,1]$, $C>0$ and $M>0$, let $\mathcal{D}^{(h,C,M)}$ be the class of distributions $P$ such that $F^\ast_x(y)=P(Y\leq y|X=x)$ satisfies:
    	\begin{enumerate}
    		\item[i)] $X\in[0,1]^d$ $P_X$-a.s.;
    		\item[ii)] For all $x\in[0,1]^d$,  $\int_\bbR F^\ast_x(z)(1-F^\ast_x(z))\rmd z\leq M$;
    		\item[iii)] $\lVert F^\ast_{x'}- F^\ast_{x}\rVert_{L^2}\leq C \lVert x'-x\rVert^h$ for all $x,x' \in [0,1]^d$.
    	\end{enumerate}
    \end{definition}
    \noindent
    Conditions $i)-iii)$ in Definition~\ref{def:classD} are very similar to the conditions considered in the point regression  framework, see Theorem 5.2 in \cite{Gyorfi_2002}. In condition $i)$, $[0,1]^d$ could be replaced by any compact set of $\mathbb{R}^d$. Condition $ii)$ requires that $\overline{\mathrm{CRPS}}(F^*_x,F^*_x)$ remains uniformly bounded by $M$, which is a condition on the dispersion of the distribution $F^*_X$ since it implies that the absolute mean error (MAE) remains uniformly bounded. Condition $iii)$ is a regularity statement of the conditional distribution in the space $L^2(\mathbb{R})$.
    As an illustration, the different conditions are expressed for the Generalized Pareto distribution model in Section~\ref{subsection:example_GPD} below.
    
    \smallskip

    Our main result is the following optimal minimax rate of convergence.
        \begin{theorem}\label{th:optimal}
    	The sequence $a_n=n^{-\frac{2h}{2h+d}}$ is the optimal minimax  rate of convergence on the class $\mathcal{D}^{(h,C,M)}$.
    \end{theorem}
    
    It should be stressed that the rate of convergence $n^{-\frac{2h}{2h+d}}$ is the same as in point regression with square error, see  Theorems 3.2 and 5.2 in \citet{Gyorfi_2002} for the lower bound and upper bound, respectively.
    
    \begin{remark}
        As pointed out by a referee, conditions \textit{i)} and \textit{iii)}  together with the integrability of $Y$ imply condition \textit{ii)} for some $M>0$. However, the dispersion, as measured by $M$, plays an important role throughout the proofs and, for this reason, we keep condition \textit{ii)} in order to obtain bounds as tight as possible.
    \end{remark}

    The proof of Theorem~\ref{th:optimal}  is divided into three steps:
    \begin{enumerate}
    	\item We provide in Section \ref{subsection:knn} an explicit and non-asymptotic upper bound for the excess risk of the $k$-nearest neighbor model uniformly on the class $\mathcal{D}^{(h,C,M)}$; the upper bound is then optimized with  a suitable choice of $k=k_n$.
    	\item In Section \ref{subsection:kernel}, we obtain similar results for the kernel model.
    	\item We show in Section \ref{subsection:lower-minimax} that $a_n=n^{-\frac{2h}{2h+d}}$ is a lower minimax rate of convergence; the main argument is that it is enough to consider a binary model when both the observation $Y$ and prediction $\hat F_X$ take values in $\{0,L\}$; we deduce that in this case, the $\mathrm{CRPS}$ coincides with the mean squared error so that we can appeal to standard results on lower minimax rate of convergence for regression.
    \end{enumerate}
    Combining these three steps, we finally obtain Theorem~\ref{th:optimal} providing the optimal minimax rate of convergence of the excess risk on the class $\mathcal{D}^{(h,C,M)}$.
    All the proofs are postponed to the Appendix.

\subsection{Upper bound for the k-nearest neighbor model}\label{subsection:knn}

    The  $k$-NN method for distributional regression is defined in Equation \eqref{eq:knn-model}. Here we do not use only the mean of the nearest neighbor sample $(Y_{i:n}(x))_{1\leq i\leq k_n}$ but its entire empirical distribution. Interestingly, the tools developed to analyze the $k$-NN in point regression can be used in our distributional regression framework.
    
    \begin{proposition}\label{pp:upper_bound_knn}
    	Assume $P\in\mathcal{D}^{(h,C,M)}$ and let $\hat F_n$ be the $k$-nearest neighbor model defined by Equation~\eqref{eq:knn-model}. Then,
    	\begin{equation*}
    		\bbE_{D_n\sim P^n}[R_P(\hat F_{n})]-R_P(F^\ast) \leq \begin{cases}
    			8^{h} C^2  \left(\cfrac{k_n}{n}\right)^{h}+\cfrac{M}{k_n}& \mbox{if } d=1, \\
    			{c_d}^{h} C^2  \left(\cfrac{k_n}{n}\right)^{2h/d}+\cfrac{M}{k_n}& \mbox{if } d\geq 2,
    		\end{cases}
    	\end{equation*}
    	where $c_d=\frac{2^{3+\frac{2}{d}}(1+\sqrt{d})^2}{V_d^{2/d}}$ and $V_d$ is the volume of the unit  ball in $\bbR^d$.
    \end{proposition}
    
    Let us stress that the upper bound is non-asymptotic and holds for all fixed $n$ and $k_n$. Optimizing the upper bound in $k_n$ yields the following corollary. 
    
    \begin{corollary} \label{cor:knn}
    	Assume $P\in\mathcal{D}^{(h,C,M)}$ and consider the $k$-NN model~\eqref{eq:knn-model}.
    	\begin{itemize}
    		\item For $d=1$, the optimal choice
    		$k_n=\left(\cfrac{M}{hC^28^h}\right)^{\frac{1}{h+1}} n^{\frac{h}{h+1}}$
    		yields
    		\begin{equation*}
    			\bbE_{D_n\sim P^n}[R_P(\hat F_n)]-R_P(F^\ast) \leq B n^{-\frac{h}{h+1}}
    		\end{equation*}
    		with constant $B=C^\frac{2}{h+1} M^{\frac{h}{h+1}} 8^{\frac{h}{h+1}} \left(h^{-\frac{h}{h+1}} +  h^{\frac{1}{h+1}} \right)$.
    		\item For $d\geq 2$, the optimal choice $k_n=\left(\cfrac{Md}{2hC^2c_d^h}\right)^{\frac{d}{2h+d}} n^{\frac{2h}{2h+d}}$
    		yields
    		\begin{equation*}
    			\bbE_{D_n\sim P^n}[R_P(\hat F_{n})]-R_P(F^\ast)\leq 	B n^{-\frac{2h}{2h+d}}
    		\end{equation*}
    		with constant $B= (C^2c_d^h)^\frac{d}{2h+d} M^{\frac{2h}{2h+d}} \left( \left(\frac{d}{2h}\right)^{\frac{2h}{2h+d}} + \left(\frac{2h}{d}\right)^{\frac{d}{2h+d}} \right)$.\\
    	\end{itemize}
    \end{corollary}
    
    \subsection{Upper bound for the kernel model}\label{subsection:kernel}
    
    Kernel methods adapted to distributional regression are defined in Equation \eqref{eq:kernel-gen}. For convenience and simplicity, we develop our result for the simple uniform kernel $K(x)= \mathds{1}_{\{\|x\|\leq 1\}}$. However, it should be stressed that all the results can be extended to boxed kernels \citep[Figure 5.7 p73]{Gyorfi_2002} to the price of some extra multiplicative constants. For the uniform kernel, the estimator writes
    \begin{equation}\label{eq:kernel-model}
    	\hat F_{n,x}(z)= \cfrac{\sum_{i=1}^{n}  \mathds{1}_{\{\|X_i-x\|\leq h_n\}}\mathds{1}_{\{Y_i \leq z\}}}{\sum_{i=1}^{n}  \mathds{1}_{\{\|X_i-x\|\leq h_n\}}},
    \end{equation}
    when the denominator is non-zero and $\hat F_n(x)=\frac{1}{n}\sum_{i=1}^{n}\mathds{1}_{\{Y_i \leq z\}}$ otherwise.
    
    \begin{proposition}\label{pp:upper_bound_kernel}
    	
    	Assume $P \in \mathcal{D}^{(h,C,M)}$ and let $\hat F_n$ be the kernel model defined by Equation \eqref{eq:kernel-model}. Then,
    	\begin{equation*}
    		\bbE_{D_n\sim P^n}[R_P(\hat F_{n})]-R_P(F^\ast) \leq \tilde{c}_d \frac{2M+C^2d^h+\frac{M}{n}}{n h_n^d}+ C^2 h_n^{2h}
    	\end{equation*}
    	where $\tilde{c}_d$ only depends on $d$. 
    \end{proposition}
    
    Once again, the upper bound is non-asymptotic and holds for all fixed $n$ and $h_n$. Optimizing the upper bound in $h_n$ yields the following corollary. 
    
    \begin{corollary}\label{cor:kernel}
    	Assume $P\in\mathcal{D}^{(h,C,M)}$ and consider the kernel model~\eqref{eq:kernel-model}.
    	For any $d$, the optimal choice $$h_n=\left(\cfrac{\tilde{c}_d d (2M+C^2d^h+\frac{M}{n})}{2hC^2}\right)^{\frac{1}{2h+d}}  n^{-\frac{1}{2h+d}}$$
    	yields
    	\begin{equation*}
    		\bbE_{D_n\sim P^n}[R_P(\hat F_{n})]-R_P(F^\ast) \leq B n^{-\frac{2h}{2h+d}}
    	\end{equation*}
    	with  $$B=C^\frac{2d}{2h+d} \left(\tilde{c}_d (2M+C^2d^h+\frac{M}{n})\right)^{\frac{2h}{2h+d}} \left( \left(\frac{d}{2h}\right)^{-\frac{d}{2h+d}} +  \left(\frac{d}{2h}\right)^{\frac{2h}{2h+d}} \right).$$
    \end{corollary}
    
\subsection{Lower minimax rate of convergence}\label{subsection:lower-minimax}
    We finally compare the rates of convergence obtained in Corollaries~\ref{cor:knn} and~\ref{cor:kernel} with a lower minimax rate of convergence in order to see whether the optimal rate of convergence is achieved.

    To prove a lower bound on a class $\mathcal{D}$, it is always possible to consider a smaller class  $\mathcal{B}$. Indeed, if $\mathcal{B}\subset \mathcal{D}$, we clearly have 
    \begin{equation*}
    	\inf_{\hat F_n}\sup_{P\in\mathcal{B}} \Big\{\bbE_{D_n\sim P^n}[R_P(\hat F_n)]-R_P(F^\ast)\Big\}
    	\leq \inf_{\hat F_n}\sup_{P\in \mathcal{D}} \Big\{\bbE_{D_n\sim P^n}[R_P(\hat F_n)]- R_P(F^\ast)\Big\}
    \end{equation*}
    so that any lower minimax rate of convergence  on $\mathcal{B}$ is also a lower minimax rate of convergence on $\mathcal{D}$.
    
    To establish the lower minimax rate of convergence, we focus on the following classes of binary responses.
    \begin{definition}\ \\
    	Let $\mathcal{B}^{(h,C,L)}$ be the class of distributions of $(X,Y)$ such that:
    	\begin{enumerate}
    		\item[i)] $Y\in \{0,L\}$ and $X$ is uniformly distributed on $[0,1]^d$;
    		\item[ii)] $\lVert F^\ast_{x'}- F^\ast_{x}\rVert_{L^2}\leq C \lVert x'-x\rVert^h$ for all $x,x' \in [0,1]^d$.
    	\end{enumerate}
    \end{definition}
    Since a binary outcome $Y\in\{0,L\}$  satisfies  $\int_{\bbR} F^\ast_x(z)(1-F_x^\ast(z))\rmd z\leq L/4$, condition $ii)$ in Definition~\ref{def:classD} holds with $M\geq L/4$. Then  $\mathcal{B}^{(h,C,L)}\subset \mathcal{D}^{(h,C,M)}$ and the following lower bound established on the smaller class also holds on the larger class.

    \begin{proposition}\label{pp:lower_bound}
    	The sequence $a_n=n^{-\frac{2h}{2h+d}}$ is a lower minimax rate of convergence on the class $\mathcal{B}^{(h,C,L)}$. More precisely,
    	\begin{equation}\label{eq:lower_bound_precise}
    		\liminf_{n\to\infty}\inf_{\hat F_n} \sup_{P\in \mathcal{B}^{(h,C,L)}} \cfrac{\bbE_{D_n\sim P^n}[R_P(\hat F_n)]-R_P(F^\ast)}{C^{\frac{2d}{2h+d}} n^{-\frac{2h}{2h+d}}}\geq C_1 
    	\end{equation}
    	for some constant $C_1>0$ independent of $C$.
    \end{proposition}
    
    Combining Corollaries~\ref{cor:knn} and~\ref{cor:kernel} and Proposition~\ref{pp:lower_bound},  we can deduce that for $d\geq2$, the $k$-NN model reaches the minimax lower rate of convergence $a_n=n^{-\frac{2h}{2h+d}}$ for the class $\mathcal{D}^{(h,C,M)}$ and that the kernel model  reaches the minimax lower rate of convergence $a_n$ in any dimension $d$. This shows that this lower rate of convergence is in fact the optimal rate of convergence and proves Theorem~\ref{th:optimal}.
    
\subsection{Generalized Pareto distributions}\label{subsection:example_GPD}
    Explicit parametric formulas of the $\mathrm{CRPS}$ exist for most classical distribution families: e.g. Gaussian, logistic, censored logistic, Generalized Extreme Value, Generalized Pareto (see \citetnp{Gneiting_2005}; \citetnp{Taillardat_2016}; \citetnp{Friederichs_2012}). We focus here on the Generalized Pareto Distribution (GPD) family and we denote by  $H_{\xi,\sigma}$ the GP distribution  with shape parameter $\xi\in\mathbb{R}$ and scale parameter $\sigma>0$. Recall that it is defined, when $\xi\neq 0$, by
    \begin{equation*}
        H_{\xi,\sigma}(z)=1-\left(1+\frac{\xi z}{\sigma}\right)_{+}^{-1/\xi}, \quad z>0,
    \end{equation*}
     with the notation  $(\cdot)_+=\max(0,\cdot)$. When $\xi=0$, the standard limit by continuity is used. For $\xi<1$, the GPD has a finite first moment and the associated  $\mathrm{CRPS}$ is given by \citep{Friederichs_2012}
    \begin{align}
        &\mathrm{CRPS}\left(H_{\xi,\sigma},y\right) \label{eq:crps_gpd} \\
        =& \left(y+\cfrac{\sigma}{\xi}\right)\left(2H_{\xi,\sigma}(y)-1\right)-\cfrac{2\sigma}{\xi(\xi-1)}\left(\cfrac{1}{\xi-2}+(1-H_{\xi,\sigma}(y))\left(1+\xi\cfrac{y}{\sigma}\right)\right).\nonumber
    \end{align}
    When $Y\sim H_{\xi^\ast,\sigma^\ast}$, the expected $\mathrm{CRPS}$ is \citep{Taillardat_2022}
        \begin{align}
        &\overline{\mathrm{CRPS}}\left(H_{\xi,\sigma},H_{\xi^\ast,\sigma^\ast}\right) \label{eq:diff_crps_gpd}\\
        =&\cfrac{\sigma^\ast}{1-\xi^\ast}+\cfrac{2\sigma}{1-\xi}m_0+\cfrac{2\xi}{1-\xi}m_1+2\sigma\left(\cfrac{1}{1-\xi}-\cfrac{1}{2(2-\xi)}\right)\nonumber
    \end{align}
    with 
    \begin{equation*}
        m_0=\bbE_{Y\sim H_{\xi^\ast,\sigma^\ast}}\left[\left(1+\cfrac{\xi}{\sigma}Y\right)^{-1/\xi}\right],\quad   m_1=\bbE_{Y\sim H_{\xi^\ast,\sigma^\ast}}\left[Y\left(1+\cfrac{\xi}{\sigma}Y\right)^{-1/\xi}\right].
    \end{equation*}
    In particular,
    \begin{equation*}\label{eq:best_crps_gpd}
        \overline{\mathrm{CRPS}}\left(H_{\xi^\ast,\sigma^\ast},H_{\xi^\ast,\sigma^\ast}\right)=\cfrac{\sigma^\ast}{(2-\xi^\ast)(1-\xi^\ast)}.
    \end{equation*}

    We now consider the distributional regression framework and we illustrate the statement of the Section \ref{subsection:CRPS_distrib} on Bayes risk in the case of a Generalized Pareto regression model where $Y$ given $X=x$ follows a GPD with shape parameter $\xi^\ast(x)$ and scale parameter $\sigma^*(x)$. Then, it is possible to show that Bayes risk is equal to
    \begin{equation*}
        R_P(F^\ast)=\int_{\bbR^d} \cfrac{\sigma^*(x)}{(2-\xi^\ast(x))(1-\xi^\ast(x))}  P_X(\rmd x)
    \end{equation*}
    when $0<\xi^\ast(x)<1$ for all $x\in\bbR^d$. For a forecast in the GPD class, i.e. $F_x$ is a GPD with shape parameter $\xi(x)$ and scale parameter $\sigma(x)$, then the risk $R_P(F)$ is equal to Bayes risk if and only if  $\xi(x) = \xi^\ast(x)$ and $\sigma(x) = \sigma^*(x)$ $P_X$-a.e.\\

    In the GPD regression framework, the conditions of the classes of distributions $\mathcal{D}^{(h,C,M)}$ can be interpreted as conditions on the parameters $\xi^\ast(x)$ and $\sigma^\ast(x)$. Condition $ii)$ is equivalent to $\sigma^*(x)\leq M(2-\xi^\ast(x))(1-\xi^\ast(x))$ when $0<\xi^\ast(x)<1$, for all $x\in[0,1]^d$. The regularity condition $iii)$ holds with constants $C$ and $h$ as soon as $x\mapsto \xi^\ast(x)$ and $x\mapsto \sigma^*(x)$ are both $h$-Hölder.\\
    
    For example, the popular case were the shape parameter $\xi^\ast(x)$ and the scale parameter $\sigma^\ast(x)$ are assumed to be linearly dependent on $x$ (i.e. $\xi^\ast(x)=\xi_0+\xi_1\cdot x$ and $\sigma^\ast(x)=\sigma_0+\sigma_1\cdot x$ with $\xi_1,\sigma_1\in\bbR^d$) is in a class of distributions of Definition \ref{def:classD}.

\section{Conclusion and Discussion}\label{section:conclusion}

    We found that the optimal rate of convergence for distributional regression on $\mathcal{D}^{(h,C,M)}$ is of the same order as the optimal rate of convergence for point regression. Thus, with regard to the sample size $n$, distributional regression evaluated with the $\mathrm{CRPS}$ converges at the same rate as point regression even though the distributional estimate carries more information on the prediction of the underlying process.
    
    We have also shown that the $k$-NN method and the kernel method reach this optimal rate of convergence, respectively in dimension $d\geq2$ and in any dimension. However, these methods are not widely used in practice because of the limitations of their predictive power in moderate or high dimension $d\geq 3$ due to the curse of dimension. An extension of this work could be to study if state-of-the-art techniques reach the optimal rate of convergence obtained in this article. Random Forests \citep{Breiman_2001} methods, such as Quantile Regression Forests \citep{Meinshausen_2006} and Distributional Random Forests \citep{Cevid_2020}, appear to be natural candidates as they are based on a generalized notion of neighborhood and have been subject to recent development in weather forecast statistical postprocessing (see, e.g., \citetnp{Taillardat_2016}).
    
    The results of this article were obtained for the CRPS, which is widely used in practice, but can easily be extended to the weighted $\mathrm{CRPS}$  in its standard uses. The weighted $\mathrm{CRPS}$  is defined as
    \begin{equation*}
    	\mathrm{wCRPS}(F,y)=\int_{\mathbb{R}}(F(z)-\mathds{1}_{y\leq z})^2 w(z)\rmd z
    \end{equation*}
    with $w$ the weight chosen. The weighted $\mathrm{CRPS}$  is used to put the focus of the score in specific regions of the outcome space \citep{Gneiting_2011}. It is used in the study of extreme events by giving more weight to the extreme behavior of the distribution.
    
    Moreover, an interesting development would be to obtain similar results for rate of convergence with respect to different strictly proper scoring rules or metrics, for instance energy scores or Wasserstein distances.\\

    \textit{Acknowledgments:} The authors acknowledge the support of the French Agence Nationale de la Recherche (ANR) under reference ANR-20-CE40-0025-01 (T-REX project) and of the Energy oriented Centre of Excellence-II (EoCoE-II), Grant Agreement 824158, funded within the Horizon2020 framework of the European Union. Part of this work was also supported by the ExtremesLearning grant from 80 PRIME CNRS-INSU and the ANR project Melody (ANR-19-CE46-0011).

\bibliography{refs.bib}

\appendix

\section{Proof of Proposition \ref{pp:upper_bound_knn}}\label{appendix:1}
For the simplicity of notation, we write simply $\bbE$ for the expectation  with respect to  $(X,Y)\sim P$ and $D_n\sim P^n$. The context makes it clear enough so as to avoid confusion.
\begin{proof}
	Recall that for the CRPS, the excess risk is equal to
	\begin{equation}
		\bbE[R_P(\hat F_{n})]-R_P(F^\ast) = \bbE\left[\int_\bbR |\hat F_{n,X}(z)-F^\ast_{X}(z)|^2 \rmd z\right].\label{eq:prop4-1}
	\end{equation}
	We first estimate $\bbE[|\hat F_{n,x}(z)-F^\ast_{x}(z)|^2]$ for fixed $x\in[0,1]^d$ and $z\in\mathbb{R}$. Denote by $X_{1:n}(x),\cdots,X_{k_n:n}(x)$ the nearest neighbors of $x$ and by $Y_{1:n}(x),\ldots,Y_{k_n:n}(x)$ the associated values of the response variable. Conditionally on $X_{i:n}(x)=x_i$, $1\leq i\leq k_n$, the random variables $Y_{i:n}(x)$, $1\leq i\leq k_n$, are independent and with distribution $F^*_{x_i}$, $1\leq i\leq k_n$. This implies that, conditionally, $\hat F_{n,x}(z)$ is the average of the $k_n$ independent random variables $\mathds{1}_{\{Y_{i:n}(x)\leq z\}}$ that have a Bernoulli distribution 
	with parameter $F^*_{x_i}(z)$. Therefore, the conditional bias and variance are given by
	\begin{align*}
		\bbE[\hat F_{n,x}(z)-F^\ast_x(z)\mid X_i(x)=x_i,1\leq i\leq k_n]=\frac{1}{k_n}\sum_{i=1}^{k_n} \left(F^\ast_{x_{i}}(z)-F^\ast_x(z)\right)\\
		\mathrm{Var}[\hat F_{n,x}(z)\mid X_i(x)=x_i,1\leq i\leq k_n]=\frac{1}{k_n^2}\sum_{i=1}^{k_n}F^\ast_{x_{i}}(z)(1-F^\ast_{x_{i}}(z)).
	\end{align*}
	Adding up the squared conditional bias and variance  and integrating with respect to $X_{i:n}(x)$, $1\leq i\leq k_n$,  we obtain the mean squared error
	\begin{align*}
		&\quad \bbE\big[|\hat F_{n,x}(z)-F^\ast_{x}(z)|^2\big]\\
		& =\bbE\Big[\Big(\frac{1}{k_n}\sum_{i=1}^{k_n} \big(F^\ast_{X_{i:n}(x)}(z)-F^\ast_x(z)\big)\Big)^2\Big]+\frac{1}{k_n^2}\sum_{i=1}^{k_n}\bbE\Big[ F^\ast_{X_{i:n}(x)}(z)(1-F^\ast_{X_{i:n}(x)}(z))\Big].
	\end{align*}
	Using Jensen's inequality and integrating with respect to $P_X(\rmd x)\rmd z$, we deduce that the excess risk~\eqref{eq:prop4-1} satisfies
	\begin{align*}
		\bbE[R_P(\hat F_{n})]-R_P(F^\ast)  &\leq \frac{1}{k_n}\sum_{i=1}^{k_n} \bbE\left[ \int_\bbR (F^\ast_{X_{i:n}(X)}(z)-F^\ast_X(z))^2\rmd z\right] \\
		&\qquad +\frac{1}{k_n^2}\sum_{i=1}^{k_n}\bbE \left[ \int_\bbR F^\ast_{X_{i:n}(X)}(z)(1-F^\ast_{X_{i:n}(X)})\rmd z\right].
	\end{align*}
	Using conditions $ii)$ and $iii)$ in the definition of the class $\mathcal{D}^{(h,C,M)}$ to bound from above the first and second term respectively, we get
	\begin{align*}
		\bbE[R_P(\hat F_{n})]-R_P(F^\ast)  &\leq  \frac{C^2}{k_n}\sum_{i=1}^{k_n} \bbE\big[\| X_{i:n}(X)-X \|^{2h}\big] +\frac{M}{k_n}\\
		&\leq  C^2\bbE\big[\| X_{k_n:n}(X)-X \|^{2h}\big]	+\frac{M}{k_n},
	\end{align*}
	where the last inequality uses the fact that, by definition of nearest neighbors, the distances $\| X_{i:n}(X)-X \|$, $1\leq i\leq k_n$, are non-increasing.
	
	The last step of the proof is to use Theorem 2.4 from \citet{Biau_2015} stating that
	\begin{equation*}
		\bbE[\| X_{k_n:n}(X)-X \|^2] \leq \begin{cases}
			8 \cfrac{k_n}{n} & \mbox{if } d=1,\\
			c_d \left(\cfrac{k_n}{n}\right)^{2/d} & \mbox{if } d\geq 2.
		\end{cases}
	\end{equation*}	
	Together with the concavity inequality (as $h\in (0,1]$)
	\begin{equation*}
		\bbE[\| X_{k_n:n}(X)-X \|^{2h}]\leq \bbE[\| X_{k_n:n}(X)-X \|^{2}]^h,
	\end{equation*}
	we deduce
	\begin{equation*}
		\bbE[R_P(\hat F_{n})]-R_P(F^\ast) \leq \begin{cases}
			C^2 8^{h} \left(\cfrac{k_n}{n}\right)^{h}+\cfrac{M}{k_n}& \mbox{if } d=1,\\
			C^2 {c_d}^{h} \left(\cfrac{k_n}{n}\right)^{2h/d}+\cfrac{M}{k_n}& \mbox{if } d\geq 2,			
		\end{cases}
	\end{equation*}
	concluding the proof of Proposition~\ref{pp:upper_bound_knn}.
\end{proof}

\section{Proof of Proposition~\ref{pp:upper_bound_kernel}}\label{appendix:2}
\begin{proof}
	Equation~\eqref{eq:kernel-model} can be rewritten as
	\[
	\hat F_{n,x}(z)= \cfrac{\sum_{i=1}^{n}  \mathds{1}_{\{X_i\in S_{x,h_n}\}}\mathds{1}_{\{Y_i \leq z\}}}{nP_n(S_{x,h_n})},
	\]
	with $S_{x,\epsilon}$ the closed ball centered at $x$ of radius $\epsilon>0$ and 
	\[
	P_n(\cdot)=\frac{1}{n}\sum_{i=1}^{n}\mathds{1}_{\{X_i\in \cdot\}}
	\]
	the empirical measure corresponding to $X_1,\dots,X_n$. Recall that we use the estimator $\hat F_n(x)=\frac{1}{n}\sum_{i=1}^{n}\mathds{1}_{\{Y_i \leq z\}}$ when $nP_n(S_{x,h_n})=0$.

	Similarly as in the proof of the Proposition \ref{pp:upper_bound_knn}, a bias/variance decomposition of the squared error yields
	\begin{align*}
		&\quad \bbE\big[|\hat F_{n,x}(z)-F^\ast_{x}(z)|^2\big]\\
		& =\bbE\left[\left( \frac{\sum_{i=1}^{n} \big(F^\ast_{X_{i}(x)}(z)-F^\ast_x(z)\big) \mathds{1}_{\{X_i\in S_{x,h_n}\}}}{nP_n(S_{x,h_n})} \right)^2 \mathds{1}_{\{nP_n(S_{x,h_n})>0\}}\right]\\
		&\quad +\bbE\left[\frac{\sum_{i=1}^{n} F^\ast_{X_{i}}(z)(1-F^\ast_{X_{i}}(z))\mathds{1}_{\{X_i\in S_{x,h_n}\}}}{(nP_n(S_{x,h_n}))^2}\mathds{1}_{\{nP_n(S_{x,h_n})>0\}}\right]\\
		&\quad+\bbE\left[\left(\frac{1}{n}\sum_{i=1}^n\mathds{1}_{\{Y_i\leq z\}}-F^\ast_x(z)\right)^2 \mathds{1}_{\{nP_n(S_{x,h_n})=0\}}\right]\\
		& := A_1(z)+A_2(z)+A_3(z).
	\end{align*}
	The excess risk at $X=x$ is thus decomposed into three terms
	\[
	\bbE\left[\int_\bbR|\hat F_{n,x}(z)-F^\ast_{x}(z)|^2\rmd z\right]=\int_\bbR A_1(z)\rmd z +\int_\bbR A_2(z)\rmd z+\int_\bbR A_3(z)\rmd z
	\]
	that we analyze successively.
	
	The first term (bias) is bounded from above using Jensen's inequality and property $iii)$ of $\mathcal{D}^{(h,C,M)}$:
	\begin{align*}
		\int_\bbR A_1(z) \rmd z &\leq \bbE\left[ \frac{\sum_{i=1}^{n} \int_\bbR \big(F^\ast_{X_{i}(x)}(z)-F^\ast_x(z)\big)^2 \rmd z \mathds{1}_{\{X_i\in S_{x,h_n}\}}}{nP_n(S_{x,h_n})} \mathds{1}_{\{nP_n(S_{x,h_n})>0\}}\right]\\
		&\leq \bbE\left[ \frac{\sum_{i=1}^{n} C^2 \| X_i-x\|^{2h}
			\mathds{1}_{\{X_i\in S_{x,h_n}\}}}{nP_n(S_{x,h_n})} \mathds{1}_{\{nP_n(S_{x,h_n})>0\}}\right]\\
		& \leq C^2 {h_n}^{2h}. 
	\end{align*}
	
	The second term (variance) is bounded using property $ii)$ of $\mathcal{D}^{(h,C,M)}$ and an elementary result for the binomial distribution:
	\begin{align*}
		\int_\bbR A_2(z) \rmd z &= \bbE\left[\frac{\sum_{i=1}^{n} \int_\bbR F^\ast_{X_{i}}(z)(1-F^\ast_{X_{i}}(z))\rmd z\mathds{1}_{\{X_i\in S_{x,h_n}\}}}{(nP_n(S_{x,h_n}))^2}\mathds{1}_{\{nP_n(S_{x,h_n})>0\}}\right]  \\
		&\leq M \bbE\left[\frac{\mathds{1}_{\{nP_n(S_{x,h_n})>0\}}}{nP_n(S_{x,h_n})} \right]\\
		&\leq \frac{2M}{nP_X(S_{x,h_n})}.
	\end{align*}
	In the last line, we use  that $Z=nP_n(S_{x,h_n})$  follows a binomial distribution with parameters $n$ and $p=P_X(S_{x,h_n})$ so that $\bbE\left[\frac{1}{Z}\mathds{1}_{\{Z>0\}}\right]\leq \frac{2}{(n+1)p}$, see  Lemma 4.1 in \citet{Gyorfi_2002}.
	
	The last term is a remainder term and is bounded by
	\begin{align*}
		\int_\bbR A_3(z) \rmd z 
		&\leq \bbE\left[\frac{1}{n}\sum_{i=1}^n \int_\bbR\left(F^\ast_{X_i}(z)-F^\ast_x(z)\right)^2\rmd z \mathds{1}_{\{nP_n(S_{x,h_n})=0\}}\right]\\
		&+ \bbE\left[\frac{1}{n^2}\sum_{i=1}^n\int_\bbR F^\ast_{X_i}(z)(1-F^\ast_{X_i}(z))\rmd z \mathds{1}_{\{nP_n(S_{x,h_n})=0\}}\right].
	\end{align*}
	Properties $ii)$ and $iii)$ of $\mathcal{D}^{(h,C,M)}$ and the fact that $\lVert X_i-x\rVert\leq\sqrt{d}$ imply
	\begin{align*}
		\int_\bbR  A_3(z) \rmd z &\leq \left(C^2d^h+\frac{M}{n}\right)\bbE\left[\mathds{1}_{\{nP_n(S_{x,h_n})=0\}}\right]\\
		&\leq \left(C^2d^h+\frac{M}{n}\right) e^{-nP_{X}(S_{x,h_n})}.
	\end{align*}
	For the second inequality, we use that  $\mathbb{P}(Z=0)=(1-p)^n\leq e^{-np}$ where $Z=nP_n(S_{x,h_n})$  follows a binomial distribution with parameters $n$ and $p=P_X(S_{x,h_n})$ .
	
	Collecting the three terms, we obtain the following upper bound for the excess risk at $X=x$:
	\begin{equation*}
	    \bbE\left[\int_\bbR |\hat F_{n,x}(z)-F^\ast_{x}(z)|^2  \rmd z\right] \leq C^2 {h_n}^{2h}+\frac{2M}{nP_X(S_{x,h_n})}+\left(C^2d^h+\frac{M}{n}\right) e^{-nP_{X}(S_{x,h_n})}.
	\end{equation*}
	
	We finally integrate this bound with respect to $P_X(\rmd x)$. According to  Equation (5.1) in \cite{Gyorfi_2002}, there exists a constant $\tilde{c}_d$ depending only on $d$ such that
	\begin{equation*}
		\int_{[0,1]^d} \frac{1}{nP_X(S_{x,h_n})} P_{X}(\rmd x)\leq \frac{\tilde{c}_d}{nh_n^d}.
	\end{equation*}
	Note that $\tilde{c}_d$ can be chosen as $\tilde{c}_d=d^{d/2}$. We also have
	\begin{align*}
		\int_{[0,1]^d}   e^{-nP_{X}(S_{x,h_n})} P_{X}(\rmd x) &\leq \max_{u\geq 0} ue^{-u} \int_{[0,1]^d} \frac{1}{nP_X(S_{x,h_n})}P_{X}(\rmd x)\\
		&\leq \frac{\tilde{c}_d}{nh_n^d} .
	\end{align*}
	We obtain thus
	\begin{align*}
		\bbE[R_P(\hat F_{n})]-R_P(F^\ast) &= \bbE\left[\int_\bbR |\hat F_{n,x}(z)-F^\ast_{x}(z)|^2  \rmd z\right] \\
		& \leq C^2 {h_n}^{2h}+\tilde{c}_d \frac{2M+C^2d^h+\frac{M}{n}}{n{h_n}^d} .
	\end{align*}
\end{proof}

\section{Proof of Proposition~\ref{pp:lower_bound}}\label{appendix:3}

The proof of  Proposition~\ref{pp:lower_bound} relies on the next two elementary lemmas. The first one states that for a binary outcome $Y\in\{0,L\}$, forecasters should focus on binary forecast $F\in\mathcal{M}(\{0,L\})$ only, which is very natural. More precisely, any predictive distribution $F\in\mathcal{M}(\bbR)$ can be associated with  $F\in\mathcal{M}(\{0,L\})$ with a better expected $\mathrm{CRPS}$.
\begin{lemma}\label{lemma:01} 
	Let $G\in\mathcal{M}(\{0,L\})$. For $F\in \mathcal{M}(\bbR)$, the distribution 
	\begin{equation*}
		\tilde{F}(z)=(1-m)\mathds{1}_{0\leq z}+m\mathds{1}_{L\leq z}\ \mbox{with } m=\frac{1}{L}\int_0^L (1-F(z)) \rmd z
	\end{equation*}
	satisfies 
	\begin{equation*}
		\overline{\mathrm{CRPS}} (\tilde F,G) \leq \overline{\mathrm{CRPS}}(F,G).
	\end{equation*}
\end{lemma}
\begin{proof}
	Let $F\in\mathcal{M}(\mathbb{R})$ and $G\in\mathcal{M}(\{0,L\})$. We have
	\begin{align*}
		\overline{\mathrm{CRPS}} (F,G) &= \int_\bbR \int_\bbR (F(z)-\mathbbm{1}_{y\leq z})^2 \rmd z G(\rmd y)\\
		&\geq \int_\bbR \int_0^L (F(z)-\mathbbm{1}_{y\leq z})^2 \rmd z G(\rmd y)
	\end{align*}
	Because $1-m$ is the mean value of $F$ on $[0,L]$, we have for $y\in\{0,L\}$
	\begin{equation*}
		\int_0^L (F(z)-\mathbbm{1}_{y\leq z})^2 \rmd z
		\geq \int_0^L ((1- m)-\mathbbm{1}_{y\leq z})^2 \rmd z.  
	\end{equation*}
	Integrating with respect to $G(\rmd y)$, we deduce
	\begin{equation*}
		\overline{\mathrm{CRPS}}(F,G)\geq \int_\bbR \int_0^L ((1- m)-\mathbbm{1}_{y\leq z})^2 \rmd z G(\rmd y).
	\end{equation*}
	The right-hand side equals $\overline{\mathrm{CRPS}}(\tilde F,G)$ and we conclude
	\begin{equation*}
		\overline{\mathrm{CRPS}}(F,G)\geq \overline{\mathrm{CRPS}}(\tilde F,G).
	\end{equation*}
\end{proof}

Lemma \ref{lemma:02} shows that for binary outcome and predictions, the $\mathrm{CRPS}$ reduces to a quantity proportional to the Brier score (\citetnp{Brier_1950})
\begin{equation*}
	\mathrm{Brier}(p,y)=(y-p)^2,\quad y\in\{0,1\}, p\in[0,1],
\end{equation*}
which is closely related to the mean squared error used in regression.

\begin{lemma}\label{lemma:02}
	For all $y\in\{0,L\}$ and $F(z)=(1-p)\mathds{1}_{0\leq z}+p\mathds{1}_{L\leq z}\in\mathcal{M}(\{0,L\})$ with $p\in[0,1]$, it holds
	\begin{equation*}
		\mathrm{CRPS}(F,y)=L\mathrm{Brier}(p,\frac{y}{L})=L(\frac{y}{L}-p)^2.
	\end{equation*}
\end{lemma}	
\begin{proof}
	We compute 
	\begin{align*}
		\mathrm{CRPS}(F,y)&=\int_0^L (1-p-\mathds{1}_{y\leq z})^2 \rmd z\\
		&= \left\{\begin{array}{ll} 
			L p^2 & \mbox{if y=0} \\
			L(1-p)^2 & \mbox{if y=L}
		\end{array}\right..
	\end{align*}
	In both cases, this equals  $L(\frac{y}{L}-p)^2=L\mathrm{Brier}(p,\frac{y}{L})$.
\end{proof}

\begin{proof}[Proof of Proposition~\ref{pp:lower_bound}]
	Since only binary outcomes are considered in the class $\mathcal{B}^{(h,C,L)}$, Lemma~\ref{lemma:01} implies that 
	\begin{equation*}
		\inf_{\hat F_n} \sup_{P\in \mathcal{B}^{(h,C,L)}} \Big\{\bbE[R_P(\hat F_n)]-R_P(F^\ast)\Big\} = \inf_{\tilde F_n} \sup_{P\in \mathcal{B}^{(h,C,L)}} \Big\{\bbE[R_P(\tilde F_n)]-R_P(F^\ast)\Big\}
	\end{equation*}
	where the infimum are taken  over models $\hat F_n$ and $\tilde F_n$ trained on the first observations $(X_i,Y_i)_{1\leq i\leq n}$ and with values  in $\mathcal{M}(\mathbb{R})$ and 
	$\mathcal{M}(\{0,L\})$, respectively. Indeed, the left-hand side is a priori smaller since the family $\hat F_n$ is larger but Lemma~\ref{lemma:01} ensures that each model $\hat F_n$ can be associated with  a model $\tilde F_n$ with equal or lower expected score.
	
	We then apply Lemma~\ref{lemma:02}. For a binary outcome, the conditional distribution of $Y$ given $X=x$ writes
	\begin{equation*}
	    F^\ast_x(z)=(1-m(x))\mathds{1}_{0\leq z}+m(x)\mathds{1}_{L\leq z},
	\end{equation*} 
	and the model $\tilde F_n$ with values in $\mathcal{M}(\{0,L\})$ takes the form 
	\begin{equation*}
		\tilde F_{n,x}(z)=(1-m_n(x))\mathds{1}_{0\leq z}+m_n(x)\mathds{1}_{L\leq z},
	\end{equation*}
	with $m(x)=\frac{1}{L}\int_0^L (1-F^\ast_x(z)) \rmd z$ and $m_n(x)=\frac{1}{L}\int_0^L (1-\hat F_{n,x}(z)) \rmd z$.\\
	Then Lemma~\ref{lemma:02} implies	
	\begin{align*}
		\bbE[R_P(\hat F_n)]-R_P(F^\ast) &= \bbE\left[\mathrm{CRPS}(\hat F_{n,X},Y)-\mathrm{CRPS}(F^*_X,Y)\right]\\
		&= L\bbE\left[(Y/L-m_{n}(X))^2-(Y/L-m(X))^2 \right]\\
		&= L\bbE\left[(m_n(X)-m(X))^2\right],
	\end{align*}
	which corresponds to the excess risk in regression with squared error loss. The property $iii)$ of $\mathcal{B}^{(h,C,L)}$ is equivalent to 
	\begin{equation*}
		|m(x)-m(x')|^h\leq C\|x-x'\|^h,\quad x\in[0,1]^d,
	\end{equation*}
	which is the standard regularity assumption on the regression function $m$. Using the result of the Problem 3.3 in \citet{Gyorfi_2002} dealing with binary models, we finally obtain that the sequence $a_n=n^{-\frac{2h}{2h+d}}$ is a lower minimax rate of convergence for this class of distributions and more precisely that Equation~\eqref{eq:lower_bound_precise} holds.
\end{proof}

\end{document}